\documentclass[notitlepage,a4paper,twoside,leqno,11pt,showkeys]{revtex4-1}
\usepackage{hyperref}
\hypersetup{colorlinks=true,allcolors=[rgb]{0,0,0.6}}
\usepackage[utf8]{inputenc}
\usepackage{color}
\definecolor{darkblue}{rgb}{0,0,0.6}
\usepackage{calc}
\usepackage[english]{babel}
\usepackage{amssymb,amsmath,amsfonts,amsthm}
\usepackage{mathrsfs}
\usepackage{dsfont}
\usepackage{enumerate}
\usepackage[shadow,textwidth=12mm,textsize=tiny,obeyDraft]{todonotes}
\usepackage{newlfont}
\usepackage[framemethod=TikZ]{mdframed}
\usepackage{enumitem}
\usepackage{tgschola}
\usepackage[T1]{fontenc}
\renewcommand{\Im}{\mathrm{Im}}

\theoremstyle{plain}
\newtheorem{thm}{Theorem}[section]
\newtheorem*{thm*}{Theorem}
\newtheorem{proposition}[thm]{Proposition}

\theoremstyle{definition}

\newtheorem*{definition*}{Definition}

\newtheorem{innerassumption}{Assumption}
\newenvironment{assumption}[1]{\renewcommand\theinnerassumption{#1}\innerassumption}{\endinnerassumption}

\theoremstyle{remark}

\newtheorem{remark}[thm]{Remark}
\newtheorem*{remark*}{Remark}
\newmdenv[
  leftmargin=\leftmargin,
  rightmargin=\leftmargin,
  topline=false,
  bottomline=false,
  skipabove=\topsep,
  skipbelow=\topsep
]{mdquote}
\renewcommand\thesection{\arabic{section}}

\makeatletter
\renewcommand\p@subsection{\thesection .}
\makeatother
\begin{document}
\bibliographystyle{abbrvnat}
\title{Self-Adjointness criterion for operators in Fock spaces}
\date{\today}


\author{Marco Falconi}
\email{marco.falconi@univ-rennes1.fr}

\affiliation{IRMAR and Centre Henri Lebesgue; Université de Rennes I\\
  Campus de Beaulieu, 263 avenue du Général Leclerc\\ CS 74205, 35042
  RENNES Cedex}

\begin{abstract}
  In this paper we provide a criterion of essential self-adjointness
  for operators in the tensor product of a separable Hilbert space and
  a Fock space. The class of operators we consider may contain a
  self-adjoint part, a part that preserves the number of Fock space
  particles and a non-diagonal part that is at most quadratic with
  respect to the creation and annihilation operators. The hypotheses
  of the criterion are satisfied in several interesting applications.
\end{abstract}

\maketitle

\section{Introduction}
\label{sec:introduction}

Let $\mathscr{H}_1$, $\mathscr{H}_2$ be separable Hilbert spaces. We
consider the following space:
\begin{equation}
  \label{eq:2}
  \mathscr{H}=\mathscr{H}_1\otimes \Gamma_s(\mathscr{H}_2)\; ;
\end{equation}
where $\Gamma_s(\mathscr{K})$ is the symmetric Fock space based on
$\mathscr{K}$ \citep[see][for mathematical presentations of Fock
spaces and second quantization]{MR0493420,MR3060648,MR0044378}. The
symmetric structure of the Fock space does not play a role in the
argument: in principle it is possible to formulate the same criterion
for anti-symmetric Fock spaces
$\mathscr{H}_1\otimes\Gamma_a(\mathscr{H}_2)$. We focus on symmetric
spaces, the corresponding antisymmetric results should be deduced
without effort.

We are interested in proving a criterion of essential self-adjointness
for densely defined operators of the form:
\begin{equation}
  \label{eq:1}
  H= H_{01}\otimes 1 + 1\otimes H_{02}+H_I\; ;
\end{equation}
with suitable assumptions on $H_{01}$, $H_{02}$ and $H_I$. Operators
based on these spaces and with such structure are crucial in physics,
to describe the quantum dynamics of interacting particles and fields.

Self-adjointness of operators in Fock spaces has been widely studied,
in particular in the context of Constructive Quantum Field Theory
\citep[e.g.][]{MR947959,MR0215585,MR0210416,MR0270671,MR0266533,Se}
and Quantum ElectroDynamics
\citep[e.g.][]{nelson:1190,MR1891842,MR2097788,BFS,BFS2,BFSS,MR797278,MR1809881,MR2436496}.
A variety of advanced tools has been utilized, for even ``simple''
systems present technical difficulties to overcome: many questions
still remain unsolved.

In some favourable situations, however, it is possible to take
advantage of the peculiar structure of the Fock space and prove
essential self-adjointness with almost no effort. The idea first
appeared in a paper by \citet{MR0266533}; and the author utilized it
in \citep{falconi-phd,Ammari:2014aa} for the Nelson model with cut
off: essential self-adjointness can be proved with less assumptions
than using the Kato-Rellich Theorem (and that becomes particularly
significative in dimension two), see
Section~\ref{sec:nels-type-hamilt}. Another remarkable application is
the Pauli-Fierz Hamiltonian describing particles coupled with a
radiation field. For general coupling constants, essential
self-adjointness has been first proved in a probabilistic setting,
using stochastic integration \citep{MR1773809,MR1891842}. In this
paper we prove the same result directly in
Section~\ref{sec:pauli-fierz-hamilt}, applying the criterion
formulated in Assumptions~\ref{ass:1}, \ref{ass:2} and
Theorem~\ref{thm:1}.

In the literature, self-adjointness of operators in Fock spaces has
been studied using various tools of functional analysis: the
Kato-Rellich and functional integration arguments mentioned above are
two examples, as well as the Nelson commutator theorem
\citep[][]{MR1814991}. For each particular system, a strategy is
utilized ad hoc: the more complicated is the correlation between
$\mathscr{H}_1$ and $\Gamma_s(\mathscr{H}_2)$, the more difficult is
the strategy. We realized that, if we take suitable advantage of the
fibered structure of the Fock space, the type of interaction between
the spaces is not so relevant. This was a strong motivation to study
the problem from a general perspective. Due to the variety of possible
applications, an effort has been made to formulate the necessary
assumptions in a general form. Roughly speaking, the essential
requirement is that the part of $H_I$ that does not commute with the
number operator of $\Gamma_s(\mathscr{H}_2)$ is at most quadratic with
respect to the creation and annihilation operators. As anticipated,
the space $\mathscr{H}_1$ does not play a particular role, as long as
$H_I$ behaves sufficiently well with respect to $H_{01}$.

\subsection{Paper organization.}
\label{sec:paper-organization}

In Section~\ref{sec:defin-notat} we introduce the notation, and recall
some basic definitions of operators in Fock spaces. In
Section~\ref{sec:assumptions-h} we formulate the necessary assumptions
on the operator $H$. In Section~\ref{sec:essent-self-adjo} we prove
the criterion. In Section~\ref{sec:applications} we outline some of
the most interesting applications. Finally in
Section~\ref{sec:conclusive-remarks-1} we give some conclusive
remarks, and an extension of the criterion to semi-bounded
quartic operators.

\subsection{Definitions and notations.}
\label{sec:defin-notat}

\begin{itemize}[label=\textcolor{darkblue}{\textbullet},itemsep=10pt]
\item Let $\mathscr{K}$ be a separable Hilbert space. Then the
  symmetric Fock space $\Gamma_s(\mathscr{K})$ is defined as the
  direct sum:
  \begin{equation*}
    \Gamma_s(\mathscr{K})=\bigoplus_{n=0}^{\infty}\mathscr{K}^{\otimes_s n}\; ,
  \end{equation*}
  where $\mathscr{K}^{\otimes_s n}$ is the $n$-fold symmetric tensor
  product of $\mathscr{K}$, and $\mathscr{K}^{\otimes_s 0}:=\mathds{C}$.
\item Let $h:\mathscr{K}\supseteq D(h)\to \mathscr{K}$ be a densely
  defined self-adjoint operator on a separable Hilbert space
  $\mathscr{K}$. Its second quantization $d\Gamma(h)$ is the
  self-adjoint operator on $\Gamma_s(\mathscr{K})$ defined by
  \begin{equation*}
    d\Gamma(h)\rvert_{D(h)^{\otimes_s n}}=\sum_{k=1}^n 1\otimes\dotsm\otimes \underbrace{h}_{k}\otimes \dotsm\otimes 1\; .
  \end{equation*}
  Let $u$ be a unitary operator on $\mathscr{K}$. We define
  $\Gamma(u)$ to be the unitary operator on $\Gamma_s(\mathscr{K})$
  given by
  \begin{equation*}
    \Gamma(u)\rvert_{\mathscr{K}^{\otimes_s n}}=\bigotimes_{k=1}^n u\; .
  \end{equation*}
  If $e^{it h}$ is a group of unitary operators on $\mathscr{K}$,
  $\Gamma(e^{it h})=e^{it d\Gamma(h)}$.
\item $N:=d\Gamma (1)$ the number operator of
  $\Gamma_s(\mathscr{H}_2)$.
\item $H_0:= H_{01}\otimes 1 + 1\otimes H_{02}$; the free Hamiltonian.
\item If $X$ is a self-adjoint operator on a Hilbert space, we denote
  by $D(X)$ its domain, by $q_X(\cdot ,\cdot)$ the form associated
  with $X$ and by $Q(X)$ the form domain.
\item Let $\mathscr{K}$ be a Hilbert space;
  $\{\mathscr{K}^{(j)}\}_{j\in\mathds{N}}$ a collection of disjoint
  subspaces of $\mathscr{K}$; $X$ an operator densely defined on
  $\mathscr{K}$. We say that $\{\mathscr{K}^{(j)}\}_{j\in\mathds{N}}$
  is invariant for $X$ if $\forall j\in \mathds{N}$, $X$ maps
  $D(X)\cap \mathscr{K}^{(j)}\to \mathscr{K}^{(j)}$, and $D(X)\cap
  \mathscr{K}^{(j)}$ is dense in $\mathscr{K}^{(j)}$.
\item Let $\mathscr{K}$ be a Hilbert space;
  $\{\mathscr{K}^{(j)}\}_{j\in\mathds{N}}$ a collection of disjoint
  closed subspaces of $\mathscr{K}$ such that
  $\bigoplus_{j\in\mathds{N}}\mathscr{K}^{(j)}=\mathscr{K}$. Then we
  call the collection complete, and we define the dense subset
  $f_0(\mathscr{K}^{(\cdot)})$ of $\mathscr{K}$ as:
  \begin{equation}
    \label{eq:3}
    f_0(\mathscr{K}^{(\cdot)})=\Bigl\{\phi\in\mathscr{K}, \exists n\in\mathds{N}\text{ s.t. }\phi\in\bigoplus_{j=0}^n\mathscr{K}^{(j)} \Bigr\}\; .
  \end{equation}
  Also, we denote by $\mathds{1}_j(\mathscr{K}^{(\cdot)})$ the
  orthogonal projection on $\mathscr{K}^{(j)}$, by $\mathds{1}_{\leq
    n}(\mathscr{K}^{(\cdot)})$ the orthogonal projection on
  $\bigoplus_{j=0}^n\mathscr{K}^{(j)}$.
\item Let $\mathscr{K}\ni f,g$ be two elements of a separable Hilbert
  space. We define the creation $a^{*}(f)$ and annihilation $a(f)$
  operators on $\Gamma_s(\mathscr{K})$ by their action on $n$-fold
  tensor products (with $a(f)\phi_0=0$ for any $\phi_0\in
  \mathscr{K}^{\otimes_s 0}=\mathds{C}$):
  \begin{align*}
    a(f)g^{\otimes n}&=\sqrt{n} \; \langle f , g\rangle_{\mathscr{K}} \; g^{\otimes (n-1)}\\
    a^{*}(f)g^{\otimes n}&=\sqrt{n+1} \; f\otimes_s g^{\otimes n}\; .
  \end{align*}
  They extend to densely defined closed operators and are adjoint of
  each other: we denote again by $a^{\#}(f)$ their closures. For any
  $f\in \mathscr{K}$, $D(a^{*}(f))=D(a(f))$ with
  \begin{equation*}
    D(a(f))=\Bigl\{\phi\in \Gamma_s(\mathscr{K})\;:\; \sum_{n=0}^{\infty}(n+1)\lVert \langle f(x) , \phi_{n+1}(x,X_n)\rangle_{\mathscr{K}(x)} \rVert_{\mathscr{K}^{\otimes_s n}(X_n)}^2<+\infty \Bigr\}\; ,
  \end{equation*}
  where $\phi_{n+1}=\phi\bigr\rvert_{\mathscr{K}^{\otimes_s n+1}}$;
  also $D(a(f))\supset D(d\Gamma(1)^{1/2})$, $D(a(f))\supset
  f_0(\mathscr{K}^{(\cdot)})$. They satisfy the Canonical
  Commutation Relations $[a(f_1),a^{*}(f_2)]=\langle f_1 ,
  f_2\rangle_{\mathscr{K}}$ on suitable domains
  (e.g. $f_0(\mathscr{K}^{(\cdot)})$).
\item We decompose $\Gamma_s(\mathscr{H}_2)$ in its subspaces with
  fixed number of particles as usual: $\forall n\in \mathds{N}$,
  define $\mathscr{H}_2^{(n)}:=\mathscr{H}_2^{\otimes_s n}$, with the
  convention $\mathscr{H}_2^{(0)}=\mathds{C}$. Then
  $\{\mathscr{H}^{(n)}_{2}\}_{n\in\mathds{N}}$ is a complete
  collection of closed disjoint subspaces of
  $\Gamma_{s}(\mathscr{H}_2)$ invariant for $N$.
\item Let $X$ be an operator on $\mathscr{H}$. We say that $X$ is
  diagonal if
  $\{\mathscr{H}_1\otimes\mathscr{H}^{(n)}_{2}\}_{n\in\mathds{N}}$ is
  invariant for $X$; $X$ is non-diagonal if for all $n\in\mathds{N}$
  and $\phi\in D(X)\cap \mathscr{H}_1\otimes\mathscr{H}^{(n)}_{2}$,
  $X\phi \notin \mathscr{H}_1\otimes\mathscr{H}^{(n)}_{2}$.
\end{itemize}

\section{Assumptions on $H$}
\label{sec:assumptions-h}

In this section we discuss Assumptions~\ref{ass:1}
and~\ref{ass:2}(\ref{ass:3}). In Section~\ref{sec:applications} below
they are checked in concrete examples.

We recall that our Hilbert space $\mathscr{H}$ has the form
\begin{equation*}
  \mathscr{H}=\mathscr{H}_1\otimes \Gamma_s(\mathscr{H}_2)\; ;  
\end{equation*}
while the operator is
\begin{equation*}
  H= H_{01}\otimes 1 + 1\otimes H_{02}+H_I\; .
\end{equation*}

We separate the assumptions on $H_0$ from the ones on $H_I$, to
improve readability. On $H_I$ we require either Assumption~\ref{ass:2}
or Assumption~\ref{ass:3}. In \ref{ass:2} the non-diagonal part of
$H_I$ can be more singular: that restricts the diagonal part to be at
most quadratic in the creation and annihilation operators. In
\ref{ass:3} on the other hand is assumed more regularity on the
non-diagonal part of $H_I$, allowing for a more singular diagonal
part.

\begin{assumption}{A$_{0}$}
  \label{ass:1}
  $H_{01}$ and $H_{02}$ are semi-bounded self-adjoint operators. We
  denote respectively by $-M_1$ and $-M_2$ their lower
  bounds. Furthermore, $\forall t\in \mathds{R}$,
  $\{\mathscr{H}_2^{(n)}\}_{n\in\mathds{N}}$ is invariant for $e^{it
    H_{02}}$.
\end{assumption}

This is quite natural. In physical systems the Hamiltonian is often
split in a part describing the free dynamics (usually a
self-adjoint and positive unbounded operator), and an interaction
part. The invariance of the $n$-particles subspaces is also a usual
feature of free quantum theories: let $h_{02}$ be a semi-bounded
self-adjoint operator on the one-particle space $\mathscr{H}_2$; then
the second quantization $d\Gamma(h_{02})$ is self-adjoint, and the
group $\Gamma(e^{it h_{02}})$ generated by it satisfies the
assumption.

\begin{assumption}{A$_I$}
  \label{ass:2}
  $H_I$ is a symmetric operator on $\mathscr{H}$, with a domain of
  definition $D(H_I)$ such that $D(H_0)\cap D(H_I)$ is dense in
  $\mathscr{H}$. Furthermore $\forall\phi\in Q(H_{01}\otimes 1)\cap Q(1\otimes H_{02})\cap
  \mathscr{H}_1\otimes \mathscr{H}_2^{(n)}$,
  \begin{equation}
    \label{eq:5}
    H_I\,\phi\in \bigoplus_{i=-2}^{2}\mathscr{H}_1\otimes \mathscr{H}_2^{(n+i)}\; .
  \end{equation}
  Also, $H_I$ satisfies the following bound: $\forall
  n\in\mathds{N}$ $\exists C>0$ such that $\forall\psi\in
  \mathscr{H}$, $\forall\phi\in Q(H_{01}\otimes 1)\cap Q(1\otimes
  H_{02})\cap\mathscr{H}_1\otimes \mathscr{H}_2^{(n)}$:
  \begin{equation}
    \label{eq:6}
    \begin{split}
      \lvert \langle \psi , H_I\phi \rangle_{\mathscr{H}}\rvert_{}^2\leq C^2\sum_{i=-2}^2\lVert \psi_{n+i} \rVert_{\mathscr{H}_1\otimes\mathscr{H}_2^{(n+i)}}^2 \Bigl[(n+1)^2\lVert \phi \rVert_{\mathscr{H}_1\otimes\mathscr{H}_2^{(n)}}^{2}+(n+1)\Bigl(q_{H_{01}\otimes 1}(\phi,\phi)\\+q_{1\otimes H_{02}}(\phi,\phi)+(\lvert M_1\rvert_{}^{}+\lvert M_2\rvert_{}^{}+1)\lVert \phi \rVert_{\mathscr{H}_1\otimes\mathscr{H}_2^{(n)}}^{2} \Bigr)\Bigr]\; ; 
    \end{split}
  \end{equation}
  where we define
  $\psi_n:=1\otimes
  \mathds{1}_n(\mathscr{H}_2^{(\cdot)})\psi$.
\end{assumption}

Consider Assumption~\ref{ass:2}. First of all, $H_I$ has to be
sufficiently regular, i.e. relatively bounded by $H_0$ (in some sense)
when restricted to the subspaces
$\mathscr{H}_1\otimes \mathscr{H}_2^{(n)}$. Essentially, we require
that $H_I$ is at most quadratic in the annihilation and creation
operators, as reflected by the $n$-dependence in \eqref{eq:6}.
\begin{assumption}{A$_I^{\prime}$}
  \label{ass:3}
  $H_I$ is a symmetric operator on $\mathscr{H}$, with a domain of
  definition $D(H_I)$ such that $D(H_0)\cap D(H_I)$ is dense in
  $\mathscr{H}$. Furthermore $\forall\phi\in Q(H_{01}\otimes 1)\cap Q(1\otimes H_{02})\cap
  \mathscr{H}_1\otimes \mathscr{H}_2^{(n)}$,
  \begin{equation}
    \label{eq:24}
    H_I\,\phi\in \bigoplus_{i=-2}^{2}\mathscr{H}_1\otimes \mathscr{H}_2^{(n+i)}\; .
  \end{equation}
  Also, $H_I=H_{diag}+H_2$ with the following properties:
  \begin{enumerate}[label=\textcolor{darkblue}{\roman*)},itemsep=10pt]
  \item \label{item:1} $H_{diag}$ is diagonal; $H_2$ is non-diagonal.
  \item\label{item:2} $H_{diag}$ satisfies the following bound.
    $\forall n\in \mathds{N}$ $\exists C(n)>0$ such that
    $\forall\psi \in\mathscr{H}$, $\forall\phi\in Q(H_{01}\otimes
    1)\cap Q(1\otimes H_{02})\cap\mathscr{H}_1\otimes
    \mathscr{H}_2^{(n)}$:
    \begin{equation}
      \label{eq:16}
      \begin{split}
        \lvert \langle \psi , H_{diag}\phi \rangle_{\mathscr{H}}\rvert_{}^2\leq C^2(n)\lVert \psi_{n} \rVert_{\mathscr{H}_1\otimes\mathscr{H}_2^{(n)}}^2 \Bigl(q_{H_{01}\otimes 1}(\phi,\phi)+q_{1\otimes H_{02}}(\phi,\phi)+(\lvert M_1\rvert_{}^{}+\lvert M_2\rvert_{}^{}\\+1)\lVert \phi \rVert_{\mathscr{H}_1\otimes\mathscr{H}_2^{(n)}}^{2} \Bigr)\; .  
      \end{split}
    \end{equation}
  \item\label{item:3} $H_2$ satisfies the following bound. $\forall
    n\in \mathds{N}$ $\exists C>0$ such that $\forall\psi
    \in\mathscr{H}$, $\forall\phi\in \mathscr{H}_1\otimes
    \mathscr{H}_2^{(n)}$:
    \begin{equation}
      \label{eq:19}
      \begin{split}
        \lvert \langle \psi , H_2\phi \rangle_{\mathscr{H}}\rvert\leq C(n+1) \lVert \phi \rVert_{\mathscr{H}_1\otimes\mathscr{H}_2^{(n)}}\sum_{\substack{i=-2 \\ i\neq 0}}^2\lVert \psi_{n+i} \rVert_{\mathscr{H}_1\otimes\mathscr{H}_2^{(n+i)}} \; . 
      \end{split}
    \end{equation}
  \end{enumerate}
\end{assumption}

Assumption~\ref{ass:3} is similar to Assumption~\ref{ass:2}. However
since the non-diagonal quadratic part $H_2$ is more regular than
before, we can be less demanding on the diagonal part $H_{diag}$: it
has still to be bounded in a suitable sense by $H_0$, but it can be
non-quadratic with respect to the creation and annihilation operators.

\begin{remark}
  In some applications, there is a decomposition of $\mathscr{H}_{1}$
  invariant for $H$. For example, it may happen that $\mathscr{H}_1$
  is also a Fock space but $H$ leaves invariant each sector with fixed
  number of particles. In this situation, we can prove essential
  self-adjointness with little less regularity on the assumptions. In
  particular, Assumption~\ref{ass:2} would be changed in:
  \begin{mdquote}
  $H_I$ is a symmetric operator on $\mathscr{H}$, with a domain of
  definition $D(H_I)$ such that $D(H_0)\cap D(H_I)$ is dense in
  $\mathscr{H}$. Furthermore there exists a complete collection
  $\{\mathscr{H}_{1}^{(j)}\otimes \Gamma_s(\mathscr{H}_2)
  \}_{j\in\mathds{N}}$
  invariant for $H_0$ and $H_I$ such that:
  $\forall\phi\in Q(H_{01}\otimes 1)\cap Q(1\otimes H_{02})\cap
  \mathscr{H}_1^{(j)}\otimes \mathscr{H}_2^{(n)}$,
  \begin{equation*}
    H_I\,\phi\in \bigoplus_{i=-2}^{2}\mathscr{H}_1^{(j)}\otimes\mathscr{H}_2^{(n+i)}\; .
  \end{equation*}
  Also, $H_I$ satisfies the following bound:
  $\forall j,n\in\mathds{N}$ $\exists C(j)>0$ such that
  $\forall\psi\in \mathscr{H}$,
  $\forall\phi\in Q(H_{01}\otimes 1)\cap Q(1\otimes
  H_{02})\cap\mathscr{H}_1^{(j)}\otimes \mathscr{H}_2^{(n)}$:
  \begin{equation*}
    \begin{split}
      \lvert \langle \psi , H_I\phi\rangle_{\mathscr{H}}\rvert_{}^2\leq C^2(j)\sum_{i=-2}^2\lVert\psi_{j,n+i}\rVert_{\mathscr{H}_1^{(j)}\otimes\mathscr{H}_2^{(n+i)}}^2\Bigl[(n+1)^2\lVert \phi\rVert_{\mathscr{H}_1^{(j)}\otimes\mathscr{H}_2^{(n)}}^{2}+(n+1)\Bigl(q_{H_{01}\otimes 1}(\phi,\phi)\\+q_{1\otimes H_{02}}(\phi,\phi)+(\lvert M_1\rvert_{}^{}+\lvert M_2\rvert_{}^{}+1)\lVert \phi\rVert_{\mathscr{H}_1^{(j)}\otimes\mathscr{H}_2^{(n)}}^{2}\Bigr)\Bigr]\; ;
    \end{split}
  \end{equation*}
  where we define
  $\psi_{j,n}:=\mathds{1}_j(\mathscr{H}_1^{(\cdot)})\otimes
  \mathds{1}_n(\mathscr{H}_2^{(\cdot)})\psi$.
  \end{mdquote}
  \begin{samepage}
    Theorem~\ref{thm:1} would then read:
  \begin{mdquote}
    Assume \ref{ass:1} and \ref{ass:2}(\ref{ass:3}). Then $H$ is
    essentially self adjoint on
    $D(H_{01}\otimes 1)\cap D(H_{02}\otimes 1)\cap
    f_0(\mathscr{H}_1^{(\cdot)}\otimes \mathscr{H}_2^{(\cdot)})$.
  \end{mdquote}
  \end{samepage}
\end{remark}

\section{Direct proof of self-adjointness}
\label{sec:essent-self-adjo}

In this section we present the criterion of essential self-adjointness
. The strategy is to prove that $\mathrm{Ran}(H\pm i)$ is dense in
$\mathscr{H}$, by an argument of reductio ad absurdum. As already
discussed, the non-diagonal part of $H_I$ is at most quadratic with
respect to the annihilation and creation operators of
$\Gamma_s(\mathscr{H}_2)$, and that plays a crucial role in the
proof. We prove Theorem~\ref{thm:1} assuming~\ref{ass:2}; the other
case being analogous.

\begin{thm}
  \label{thm:1}
  Assume \ref{ass:1} and \ref{ass:2}(\ref{ass:3}). Then $H$ is
  essentially self adjoint on $D(H_{01}\otimes 1)\cap D(H_{02}\otimes
  1)\cap \mathscr{H}_1\otimes f_0( \mathscr{H}_2^{(\cdot)})$.
\end{thm}
\begin{proof}
  Let $\psi\in\mathscr{H}$, $z\in \mathds{C}$ with $\Im z\neq
  0$. Suppose that $\forall \phi\in D(H_{01}\otimes 1)\cap D(1\otimes
  H_{02})\cap \mathscr{H}_1\otimes f_0(\mathscr{H}_2^{(\cdot)})$:
  \begin{equation}
    \label{eq:7}
    \langle \psi , (H-z)\phi\rangle_{\mathscr{H}}=0\; .
  \end{equation}
  Then it suffices to show that $\psi = 0$. This is done in few
  steps. Let $n\in \mathds{N}$ and $\phi_{n}\in D(H_{01}\otimes
  1)\cap D(1\otimes H_{02})\cap \mathscr{H}_1\otimes
  \mathscr{H}_2^{(n)}$. For all $n\in\mathds{N}$, the space
  $Q(H_{01}\otimes 1)\cap Q(1\otimes H_{02})\cap
  \mathscr{H}_1\otimes \mathscr{H}_2^{(n)}$ with the scalar
  product:
  \begin{equation}
    \label{eq:8}
    \langle \,\cdot \, ,\, \cdot\,\rangle_{\mathscr{X}_{n}}= q_{H_{01}\otimes 1}(\,\cdot\, ,\,\cdot\,)+q_{1\otimes H_{02}}(\,\cdot\, ,\, \cdot\,) +(\lvert M_1\rvert_{}^{}+\lvert M_2\rvert_{}^{}+1)\langle \, \cdot \, ,\, \cdot\,\rangle_{\mathscr{H}_1\otimes \mathscr{H}_2^{(n)}}
  \end{equation}
  is complete, and therefore a Hilbert space. We denote it by
  $\mathscr{X}_{n}$. Then~\eqref{eq:7} together with
  Assumption~\ref{ass:1} imply, since $\phi_{n}\in D(H_{01}\otimes
  1)\cap D(1\otimes H_{02})$:
  \begin{equation}
    \label{eq:9}
    \langle \psi_{n}, \phi_{n} \rangle_{\mathscr{X}_{n}}= (z+\lvert M_1\rvert_{}^{}+\lvert M_2\rvert_{}^{}+1)\langle \psi_{n} ,\phi_{n}\rangle_{\mathscr{H}_1\otimes \mathscr{H}_2^{(n)}}-\langle \psi , H_I\phi_{n} \rangle_{\mathscr{H}}\; .
  \end{equation}
  Use bound~\eqref{eq:16} and then Riesz's Lemma on
  $\mathscr{X}_{n}$: it follows that $\psi_{n}\in Q(H_{01}\otimes
  1)\cap Q(1\otimes H_{02})\cap \mathscr{H}_1\otimes
  \mathscr{H}_2^{(n)}$ for any $n\in\mathds{N}$.

  Let
  $\phi\in Q(H_{01}\otimes 1)\cap Q(1\otimes H_{02})\cap
  \mathscr{H}_1\otimes f_0(\mathscr{H}_2^{(\cdot)})$.
  Then $\exists \{\phi^{(\alpha)}\}_{\alpha\in\mathds{N}}$ such that
  $\forall \alpha\in\mathds{N}$,
  $\phi^{(\alpha)}\in D(H_{01}\otimes 1)\cap D(1\otimes H_{02})\cap
  \mathscr{H}_1\otimes f_0( \mathscr{H}_2^{(\cdot)})$;
  and $\forall n\in\mathds{N}$, $\phi_{n}^{(\alpha)}\to \phi_{n}$
  in the topology induced by
  $\lVert \,\cdot\, \rVert_{\mathscr{X}_{n}}^{}$. Furthermore
  $\forall\alpha\in\mathds{N}$:
  \begin{equation}
    \label{eq:11}
    \langle \psi , (H-z)\phi^{(\alpha)} \rangle_{\mathscr{H}}=0\; .
  \end{equation}
  Since
  $\psi_{n}\in Q(H_{01}\otimes 1)\cap Q(1\otimes H_{02})\cap
  \mathscr{H}_1\otimes \mathscr{H}_2^{(n)}$,
  we can take the limit of~\eqref{eq:11} and obtain,
  $\forall \phi\in Q(H_{01}\otimes 1)\cap Q(1\otimes H_{02})\cap
  \mathscr{H}_1\otimes f_0( \mathscr{H}_2^{(\cdot)})$:
  \begin{equation}
    \label{eq:12}
    q_{H_{01}\otimes 1}(\psi,\phi)+q_{1\otimes H_{02}}(\psi,\phi)+\langle \psi , H_I\phi \rangle_{\mathscr{H}}=z\langle \psi , \phi\rangle_{\mathscr{H}}\; .
  \end{equation}
  Hence we can choose
  $\phi=\psi_{\leq n}:= 1\otimes \mathds{1}_{\leq
    n}(\mathscr{H}_2^{(\cdot)})\psi$
  in~\eqref{eq:12}. Then, using Assumption~\ref{ass:1} and taking the
  imaginary part we obtain:
  \begin{equation}
    \label{eq:15}
    \Im (z) \langle \psi_{\leq n} , \psi_{\leq n}\rangle=\Im (\langle  \psi -\psi_{\leq n}, H_{I}\psi_{\leq n} \rangle_{})\; .
  \end{equation}
  Now, by Assumption~\ref{ass:2} (the equality holds on the suitable
  domain):
  \begin{equation*}
    H_I \bigl(1\otimes \mathds{1}_{\leq n}(\mathscr{H}_2^{(\cdot)})\bigr) =\bigl(1\otimes \mathds{1}_{\leq n+2}(\mathscr{H}_2^{(\cdot)})\bigr)H_I \bigl(1\otimes \mathds{1}_{\leq n}(\mathscr{H}_2^{(\cdot)})\bigr)\; .
  \end{equation*}
  Furthermore
  $1\otimes \mathds{1}_{\leq n+2}(\mathscr{H}_2^{(\cdot)})(\psi -
  \psi_{\leq n})= \psi _{n+1}\oplus \psi _{n+2}$.
  Then Equation~\eqref{eq:15} becomes:
  \begin{equation*}
    \Im (z) \langle \psi_{\leq n} , \psi_{\leq n}\rangle= \sum_{i=1}^2\Im(\langle \psi_{n+i} , H_I\psi_{\leq n}\rangle_{})\; .
  \end{equation*}
  Using the symmetry of $H_I$, and~\eqref{eq:5} we obtain:
  \begin{equation}
    \label{eq:17}
    \Im (z) \langle \psi_{\leq n} , \psi_{\leq n}\rangle=\Im(\langle \psi_{n+2} , H_I\psi_{n}\rangle+\langle \psi_{n+1} , H_I\psi_{n}\rangle+\langle \psi_{n+1} , H_I\psi_{n-1}\rangle )\; .
  \end{equation}
  Now bound~\eqref{eq:17} using~\eqref{eq:6}; then we obtain
  $\forall n\in \mathds{N}$:
  \begin{equation}
    \label{eq:18}
    \begin{split}
      \lvert \Im z\rvert\sum_{i=0}^n\lVert \psi_{i} \rVert_{}^2\leq C\Bigl[\lVert \psi_{n+1} \rVert\Bigl((n+1)\bigl(\lVert \psi_{n} \rVert+\lVert \psi_{n-1} \rVert \bigr)+\sqrt{n+1}\bigl(\lVert \psi_{n} \rVert_{\mathscr{X}_{n}}+\lVert \psi_{n-1} \rVert_{\mathscr{X}_{n-1}}\bigr)\Bigr)\\+\lVert \psi_{n+2} \rVert\Bigl((n+1)\lVert \psi_{n} \rVert+\sqrt{n+1}\lVert \psi_{n} \rVert_{\mathscr{X}_{n}}\Bigr)\Bigr]\\\leq 2C(n+1)\Bigl[ \sum_{i_1=0}^2\lVert \psi_{n+i_1} \rVert_{}^2+\sum_{i_2=-1}^0(n+1)^{-1}\lVert \psi_{n+i_2} \rVert_{\mathscr{X}_{n+i_2}}^2 \Bigr]  \; .
    \end{split}
  \end{equation}
  
  For all $\alpha >0$ define:
  \begin{equation*}
    S:= \sum_{n=0}^{\infty}\lVert \psi_{n} \rVert_{}^2\; ; \;S_\alpha:=\sum_{n=0}^{\infty}(n+\alpha)^{-1}\lVert \psi_{n}\rVert_{\mathscr{X}_{n}}^2\; .
  \end{equation*}
  $\psi\in\mathscr{H}$, hence $S$ is finite. We prove that also
  $S_{\alpha}$ is finite. Using equation~\eqref{eq:12} with
  $\phi=\psi_{n}$ we obtain, for all $n\in\mathds{N}$:
  \begin{equation}
    \label{eq:38}
    (n+\alpha)^{-1}\lVert \psi_{n} \rVert_{\mathscr{X}_{n}}^2=(n+\alpha)^{-1}(z+\lvert M_1\rvert_{}^{}+\lvert M_2\rvert_{}^{}+1)\lVert \psi_{n} \rVert_{}^2-(n+\alpha)^{-1}\langle \psi , H_I\psi_{n}\rangle_{}\; .
  \end{equation}
  Now, we can use bound~\eqref{eq:6} on
  $(n+\alpha)^{-1}\lvert \langle \psi ,
  H_I\psi_{n}\rangle\rvert_{}^{}$, obtaining
  \begin{equation}
    \label{eq:39}
    \begin{split}
      (n+\alpha)^{-2}\lvert \langle \psi , H_I\psi_{n} \rangle_{\mathscr{H}}\rvert_{}^2\leq C^2 (n+\alpha)^{-2}\sum_{i=-2}^2\lVert \psi_{n+i} \rVert^2 \Bigl[(n+1)^2\lVert \psi_{n} \rVert^2+(n+1)\lVert \psi_{n} \rVert_{\mathscr{X}_{n}}^2\Bigr]\\\leq C^2(\alpha)\sum_{i=-2}^2\lVert \psi_{n+i} \rVert_{}^2 \Bigl[\lVert \psi_{n} \rVert_{}^2+(n+\alpha)^{-1}\lVert \psi_{n} \rVert_{\mathscr{X}_{n}}^2\Bigr]\; ,      
    \end{split}
  \end{equation}
  for some $C(\alpha)>0$. The only terms we need to deal with are
  $(n+\alpha)^{-1}\lVert \psi_{n+i} \rVert_{}^2 \lVert \psi_{n}
  \rVert_{\mathscr{X}_{n}}^2$.
  We use the fact that for any $\varepsilon,a,b>0$,
  $ab\leq \frac{1}{2}(\varepsilon a^2+\frac{1}{\varepsilon}b^2)$,
  obtaining
  \begin{equation}
    \label{eq:40}
    (n+\alpha)^{-1}\lVert \psi_{n+i} \rVert_{}^2 \lVert \psi_{n}\rVert_{\mathscr{X}_{n}}^2\leq \frac{1}{2}\Bigl(\varepsilon(n+\alpha)^{-2}\lVert \psi_{n}\rVert_{\mathscr{X}_{n}}^4+\frac{1}{\varepsilon}\lVert \psi_{n+i} \rVert_{}^4\Bigr)\; .
  \end{equation}
  Combining~\eqref{eq:40} with~\eqref{eq:39}, and applying to
  Equation~\eqref{eq:38}, we obtain the following bound: for all
  $\varepsilon,\alpha>0$, $\exists
  C(\alpha,\varepsilon)>0$ such that
  \begin{equation}
    \label{eq:41}
    (n+\alpha)^{-1}\lVert \psi_{n} \rVert_{\mathscr{X}_{n}}^2\leq C(\alpha,\varepsilon)\sum_{i=-2}^2\lVert \psi_{n+i} \rVert_{}^2+\varepsilon(n+\alpha)^{-1}\lVert \psi_{n} \rVert_{\mathscr{X}_{n}}^2\; .
  \end{equation}
  Fix $\varepsilon <1$, then for all $\alpha >0$, $\exists C(\alpha)>0$ such that $\forall \bar{n}\in
  \mathds{N}$:
  \begin{equation}
    \label{eq:4}
    \sum_{n=0}^{\bar{n}} (n+\alpha)^{-1}\lVert \psi_{n} \rVert_{\mathscr{X}_{n}}^2\leq C(\alpha) S\; ;
  \end{equation}
  uniformly in $\bar{n}$. Then we can take the limit
  $\bar{n}\to\infty$ and obtain $S_{\alpha}<\infty$.
  \begin{remark*}
    The bound of Equation~\eqref{eq:41} could seem to follow from an
    implicit smallness condition on the interaction $H_I$. As it will
    become clearer with the examples of
    Section~\ref{sec:applications}, it is not the case. Roughly
    speaking, Assumption~\ref{ass:2} allows for interaction parts that
    are at most as singular as
    $(H_0+\lvert M_1\rvert_{}^{}+\lvert
    M_2\rvert_{}^{})^{1/2}(N+1)^{1/2}$.
  \end{remark*}

  Now return to Equation~\eqref{eq:18}. There exists $n^{*}\in\mathds{N}$ such that $\forall n\geq n^{*}$:
  \begin{equation*}
    \frac{1}{2}S\leq \sum_{i=0}^{n}\lVert \psi_{i} \rVert_{}^2\leq S\; .
  \end{equation*}
  Hence summing in $n^{*} \leq n \leq \bar{n}$ on both sides
  of~\eqref{eq:18} we obtain for all $\bar{n}>n^{*}$:
  \begin{equation*}
    \begin{split}
      \frac{1}{2}S\sum_{n=n^{*}}^{\bar{n}}(n+1)^{-1}\leq \sum_{n=n^{*}}^{\bar{n}}(n+1)^{-1}\sum_{i=0}^n\lVert \psi_{i} \rVert_{}^2\leq 2 \frac{C}{\lvert \Im z\rvert_{}^{}}(3S+S_1+S_2)\; .
    \end{split}
  \end{equation*}
  The bound on the right hand side is uniform in $\bar{n}$: that is
  absurd, unless $S=S_1=S_2=0$
  $\Leftrightarrow$ $\psi =0$.
\end{proof}

Once essential self-adjointness is established, it is possible to give
the following characterization of the domain of self-adjointness
$D(H)$.

\begin{proposition}
  \label{prop:1}
  Assume \ref{ass:1} and \ref{ass:2}(\ref{ass:3}). If exists $K$
  self-adjoint operator with domain $D(K)$ such that:
  \begin{enumerate}[label=\textcolor{darkblue}{\roman*)},itemsep=10pt]
  \item $D(H_{0})\cap D(K)$ is dense in $\mathscr{H}$;
    $\mathscr{H}_1\otimes f_0(\mathscr{H}_2^{(\cdot)})$ is
    dense in $D(K)$.
  \item There exists $ 0 < \varepsilon < 1$ such that $\exists
    C(\varepsilon)>0$, $\forall \phi\in D(H_{0})\cap D(K)$:
    \begin{equation}
      \label{eq:20}
      \lVert H_I\phi \rVert_{}^{}\leq \varepsilon\lVert H_{0}\phi \rVert_{}^{}+C(\varepsilon)(\lVert K\phi \rVert_{}^{}+\lVert \phi \rVert_{}^{})\; .
    \end{equation}
  \end{enumerate}
  Then $D(H)\cap D(K)=D(H_{0})\cap D(K)$.
\end{proposition}
\begin{proof}
  Using bound~\eqref{eq:20}, we have $\forall \phi \in D(H_{0})\cap
  D(K)$:
  \begin{equation}
    \label{eq:21}
  \lVert H\phi \rVert_{}^{}\leq (\varepsilon+1)\lVert H_{0}\phi \rVert_{}^{}+C(\varepsilon)(\lVert K\phi \rVert_{}^{}+\lVert \phi \rVert_{}^{})\; .
  \end{equation}
  Then $D(H)\supseteq D(H_{0})\cap D(K)$. Now let $\phi\in D(H)\cap
  D(K)$: using~\eqref{eq:20}
  \begin{equation}
    \label{eq:22}
    \lVert H_{0}\phi \rVert_{}^{}\leq \varepsilon\lVert H_0\phi \rVert_{}^{} +\lVert H\phi \rVert_{}^{}+C(\varepsilon)(\lVert K\phi \rVert_{}^{}+\lVert \phi \rVert_{}^{})\; ;
  \end{equation}
  since $\varepsilon <1$, $D(H_0)\supseteq D(H)\cap D(K)$.
\end{proof}

\section{Applications}
\label{sec:applications}

It is possible to apply Theorem~\ref{thm:1} in several situations of
mathematical and physical interest. We present and discuss some of
them in this section; not before a brief discussion of the
``boundaries'' of Theorem~\ref{thm:1}: it may be interesting to see
how its proof fails when we consider operators that are more than
quadratic in the annihilation/creation operators; and to define a
quadratic operator that is not sufficiently regular for
Assumption~\ref{ass:2}(\ref{ass:3}) to hold. According to this
purpose, we will consider simple toy models on
$\Gamma _s(\mathds{C})$. We denote by $a^{\#}$ the corresponding
annihilation/creation operators.

Let's consider a simple trilinear Hamiltonian on $\Gamma _s(\mathds{C})$:
\begin{equation*}
  H_3=a^{*}a+a^{*}a^{*}a^{*}+aaa\; .
\end{equation*}
The free part is $H_0=a^{*}a$, and the interaction part is
$H_I=a^{*}a^{*}a^{*}+aaa$. Assumption~\ref{ass:1} is satisfied, and
Assumption~\ref{ass:3} is slightly modified: $i$ now ranges from $-3$
to $3$, and bounds~\eqref{eq:16} and~\eqref{eq:19} are replaced by the
simple bound:
\begin{equation*}
  \lvert \langle  \psi  , H_I\phi  \rangle  \rvert_{\mathscr{H}}\leq C(n+1)^{3/2}\lVert \phi   \rVert_{\mathscr{H}^{(n)}}^{}\bigl(\lVert \psi _{n+3}  \rVert_{\mathscr{H}^{(n+3)}}^{}+\lVert \psi _{n-3}  \rVert_{\mathscr{H}^{(n-3)}}^{}\bigr)\; .
\end{equation*}
The proof of Theorem~\ref{thm:1} carries on, almost unchanged, up to
Equation~\eqref{eq:18} that would now read
\begin{equation*}
  \lvert \Im z  \rvert_{}^{}\sum_{i=0}^n \lVert \psi_i   \rVert_{}^2\leq C (n+1)^{3/2}\lVert \psi _{n+3}  \rVert_{}^2\; .
\end{equation*}
However if we now take the sum in $n$ from $n^{*}$ to $\bar{n}$ (where
$n^{*}$ is such that
$\frac{1}{2}\lVert \psi \rVert_{}^2\leq \sum_{i=0}^n \lVert \psi_i
\rVert_{}^2\leq \lVert \psi \rVert_{}^2$
for all $n\geq n^{*}$) we \emph{cannot} conclude that
$\lVert \psi \rVert_{}^{}$ must be zero, because the series
$\sum_{n=0}^{\infty }(n+1)^{-3/2}$ converges. Hence the proof fails,
and analogously would fail for any higher order polynomial of the
annihilation/creation operators.

On the other hand, we introduce now a quadratic model for which
Assumption~\ref{ass:2}(\ref{ass:3}) fails to hold, and thus
Theorem~\ref{thm:1} cannot be applied. For the following operator on
$L^2 (\mathds{R}^{} )\otimes \Gamma _s(\mathds{C})$
Assumption~\ref{ass:2} \emph{is satisfied}:
\begin{equation*}
  H_{\partial  a}=-\partial  _x^2+a^{*}a-i\partial  _x(a^{*}+a)+a^{*}a^{*}+aa\; ,
\end{equation*}
where $H_0=-\partial  _x^2+a^{*}a$ and
$H_I=-i\partial  _x(a^{*}+a)+a^{*}a^{*}+aa$. If, however, the derivative
operator is coupled with the quadratic term
\begin{equation*}
  H_{\partial  aa}=-\partial  _x^2+a^{*}a-i\partial  _x(a^{*}a^{*}+aa)\; ,
\end{equation*}
\ref{ass:2}(\ref{ass:3}) is no longer satisfied. The interaction in
this case would be of type $H_0^{1/2}N$, and therefore too singular:
Theorem~\ref{thm:1} \emph{does not hold} for $H_{\partial aa}$.

Throughout the section we will adopt the following notations, in
addition to the ones of Section~\ref{sec:defin-notat}. Let
$\mathscr{K}$ a Hilbert space; we denote by $\mathcal{L}(\mathscr{K})$
the set of bounded operators on $\mathscr{K}$ and by
$\lvert \,\cdot\,\rvert_{\mathcal{L}(\mathscr{K})}^{}$ the operator
norm. It is also useful to define the annihilation/creation operator
valued distributions $a^{\#}(x)$, $x\in\mathds{R}^d$. Let
$f\in L^2 (\mathds{R}^d)$, $a^{\#}(f)$ the annihilation/creation
operators on $\Gamma_s(L^2 (\mathds{R}^d))$. Then the operator valued
distributions $a^{\#}(x)$ acting on $L^2 (\mathds{R}^d)$, with values
on $\Gamma_s(L^2 (\mathds{R}^d))$, are defined by:
\begin{equation*}
  (a^{*},f)\equiv\int_{\mathds{R}^d}a^{*}(x)f(x)dx:=a^{*}(f)\; ;\; (a,f)\equiv\int_{\mathds{R}^d}a(x)\bar{f}(x)dx:=a(f)\; .
\end{equation*}
They satisfy the commutation relations (inherited by the CCR)
$[a(x),a^{*}(y)]=\delta(x-y)$.

\subsection{Hamiltonians of identical bosons.}
\label{sec:hamilt-many-bosons}

The criterion applies to operators in the Fock space
$\Gamma_s(\mathscr{K})$, for any separable Hilbert space
$\mathscr{K}$. Simply choose $\mathscr{H}_1\equiv \mathds{C}$ and
$\mathscr{H}_2\equiv \mathscr{K}$; then $\mathds{C}\otimes
\Gamma_s(\mathscr{K})\approx \Gamma_s(\mathscr{K})$ up to an unitary
isomorphism.

An example is given by the following class of operators. Let
$\mathscr{K}=L^2 (\mathds{R}^d)$; $h_0$ a positive self adjoint
operator on $L^2 (\mathds{R}^{d})$ (the one-particle free
Hamiltonian). Furthermore, let $V_1\in L^2 (\mathds{R}^{d})$, $V_2,
V_3\in L^2 (\mathds{R}^{2d})$, with $V_2=\overline{V}_2$, and
$V_4(\cdot):\mathds{R}^d\to \mathds{R}$, such that $V_4(x)=V_4(-x)$
and $V_4\,(h_0+1)^{-1/2}\in \mathcal{L}(L^2 (\mathds{R}^d))$. Consider
\begin{equation}
  \label{eq:32}
  \begin{split}
    H=d\Gamma(h_0)+\int_{\mathds{R}^d}^{}\Bigl(V_1(x)a^{*}(x)+\overline{V}_1(x)a(x)\Bigr)  dx+ \int_{\mathds{R}^{2d}}^{}\Bigl(V_2(x,y)a^{*}(x)a(y)+V_3(x,y)a^{*}(x)\\
    a^{*}(y) +\overline{V}_3(x,y)a(x)a(y)\Bigr)  dxdy+\frac{1}{2}\int_{\mathds{R}^{2d}}^{}V_4(x-y)a^{*}(x)a^{*}(y)a(x)a(y)  dxdy\; .
  \end{split}
\end{equation}
We make the following identifications:
$H_{01}\equiv 0$, $H_{02}\equiv d\Gamma(h_0)$, $H_{diag}\equiv\int
(V_4a^{*}a^{*}aa+V_2a^{*}a)$, $H_2\equiv
\int_{}^{}(V_1a^{*}+\overline{V}_1a)+\int
(V_3a^{*}a^{*}+\overline{V}_3aa)$. Assumption~\ref{ass:1} is trivial
to verify; and Assumption~\ref{ass:3} follows from standard estimates
on Fock space: let $\psi\in\Gamma_s(L^2 (\mathds{R}^d))$, $\phi_n\in
L^2_s(\mathds{R}^{nd})\cap Q(d\Gamma(h_0))$, $n\in\mathds{N}$, then
\begin{gather}
  \label{eq:35}
  \begin{split}
    \lvert \langle  \psi, H_{diag}\phi_n \rangle_{}\rvert_{}^{}\leq \Bigl(n\lVert V_2 \rVert_2^{}\lVert \phi_n \rVert_{}^{}+\lvert V_4\,(h_{0}+1)^{-1/2}\rvert_{\mathcal{L}(L^2 (\mathds{R}^d))}^{}\bigl(n^{3/2}\lVert (d\Gamma(h_0))^{1/2}\phi_n \rVert_{}^{}\\+n^2\lVert \phi_n \rVert_{}^{}\bigr)\Bigr)\lVert \psi_n \rVert_{}^{}\; ;
  \end{split}\\
  \label{eq:36}  \lvert \langle  \psi, H_2\phi_n \rangle_{}\rvert_{}^{}\leq 2\Bigl(\sqrt{n+1}\lVert V_1 \rVert_2^{}+ (n+1)\lVert V_3 \rVert_2^{}\Bigr)\lVert \phi_n \rVert_{}^{}\sum_{\substack{i=-2\\i\neq 0}}^{2}\lVert \psi_{n+i} \rVert_{}^{}\; .
\end{gather}
Hence we can apply Theorem~\ref{thm:1}; and prove essential
self-adjointness of $H$ in $D(d\Gamma (h_0))\cap
f_0(L^2(\mathds{R}^{d})^{(\cdot)})$. We can also apply
Proposition~\ref{prop:1} with $K\equiv N^{3}$, i.e. $D(H)\cap
D(N^{3})=D(d\Gamma(h_0))\cap D(N^{3})$. Observe that if $d=3$,
the well-known many body Hamiltonian with Coulomb pair interaction
\begin{equation*}
  H_C=d\Gamma(-\Delta)\pm\frac{1}{2}\int_{\mathds{R}^6}^{}\frac{1}{\lvert x-y\rvert_{}^{}}a^{*}(x)a^{*}(y)a(x)a(y)  dxdy\; ,
\end{equation*}
is just the special case $h_0= -\Delta$, $V_1=V_2=V_3=0$ and
$V_4=\pm \lvert x \rvert_{}^{-1}$.

\subsection{Nelson-type Hamiltonians.}
\label{sec:nels-type-hamilt}

We consider now the dynamics of different species of particles (or
fields) interacting. A typical example is the Nelson Hamiltonian. It
was introduced in a rigorous way by \citet{nelson:1190} to describe
nucleons in a meson field, and studied by several authors
\citep[e.g.][]{DG1,MR1814991,MR1809881,GHPS}.

Let $\mathscr{H}=L^2 (\mathds{R}^{pd})\otimes \Gamma_s(L^2
(\mathds{R}^d))$: the first space corresponds to $n$ non-relativistic
particles; the second to a scalar relativistic field. Let $\omega$ be
a positive self-adjoint operator on $L^2 (\mathds{R}^d)$ (the
dispersion relation of the relativistic field), $V\in L^2_{loc}
(\mathds{R}^d, \mathds{R}_{+})$ an external potential acting on the
particles. The interaction between the particles and the field is
linear in the creation and annihilation operators $a^{\#}$
corresponding to the field. Let $v:\mathds{R}^{2d}\to \mathds{C}$ such
that
\begin{itemize}[label=\textcolor{darkblue}{\textbullet},itemsep=10pt]
\item $(1-\Delta_x)^{-1/2}\lVert v(x,\cdot) \rVert_{L^2_{(k)}
    (\mathds{R}^d)}^2(1-\Delta_x)^{-1/2}\in \mathcal{L}(L^2_{(x)}
  (\mathds{R}^{d}))$;
\item for all $k\in\mathds{R}^d$, $v(x,k)(1-\Delta_x)^{-1/2}\in
  \mathcal{L}(L^2_{(x)} (\mathds{R}^d))$, with $\lvert
  v(x,\cdot)(1-\Delta_x)^{-1/2}\rvert_{\mathcal{L}(L^2_{(x)}
    (\mathds{R}^{d}))}^{}\in L^2_{(k)} (\mathds{R}^d)$.
\end{itemize}
Then we define the Nelson Hamiltonian:
\begin{equation}
  \label{eq:23}
  H_N=\Bigl(\sum_{i=1}^p -\Delta_{x_i}+V(x_i)\Bigr)\otimes 1+1\otimes d\Gamma(\omega)+\sum_{i=1}^pa^{*}(v(x_{i},\cdot))+a(v(x_i,\cdot))\; .
\end{equation}
The function $v$ describes the coupling between the particles and the
relativistic field. The assumptions above imply that it has a good
behaviour both for high and small momenta; in particular in
three-dimensions it acts as an UV cutoff function.
\begin{remark*}
  The model of \citet{nelson:1190} was much more specific: $d=3$,
  $\omega(k)=\sqrt{k^2+\mu^2}$ with $\mu>0$, $V=0$ and
  $v(x,k)=\lambda(2\pi)^{-3/2}(2\omega(k))^{-1/2}e^{-ik\cdot
    x}\mathds{1}_{\lvert \,\cdot\,\rvert \leq \sigma}(k)$ with
  $\lambda,\sigma >0$. With these assumptions, $v\in L^{\infty}
  (\mathds{R}^3,L^2 (\mathds{R}^3))$, $\omega^{-1/2}v\in L^{\infty}
  (\mathds{R}^3,L^2 (\mathds{R}^3))$; then $H_N$ (the Nelson model
  with UV cut off) is self-adjoint by the Kato-Rellich
  Theorem. However, if we consider $d=2$ and $\mu=0$ (massless
  relativistic field), the Kato-Rellich Theorem is not applicable
  because $\omega^{-1/2}v\notin L^{\infty} (\mathds{R}^2,L^2
  (\mathds{R}^2))$ due to an infrared divergence. Instead assumptions
  \ref{ass:1} and \ref{ass:3} are still satisfied, thus
  Theorem~\ref{thm:1} can be used.
\end{remark*}

In order to check Assumptions~\ref{ass:1} and~\ref{ass:2}
on~\eqref{eq:23}, we make the (straightforward) identifications:
$\mathscr{H}_1\equiv L^2 (\mathds{R}^{pd})$, $\mathscr{H}_2\equiv L^2
(\mathds{R}^d)$, $H_{01}\equiv\sum_i-\Delta_{x_i}+V(x_i)$,
$H_{02}\equiv d\Gamma(\omega)$, $H_I\equiv
\sum_ia^{*}(v(x_i,\cdot))+a(v(x_i,\cdot))$. We do not need to
introduce a decomposition of $\mathscr{H}_1$. Assumption~\ref{ass:1}
is satisfied: for all $V\in L^2_{loc} (\mathds{R}^d,\mathds{R}_{+})$,
$-\Delta+V(\cdot)$ is a positive self-adjoint operator, and the
vectors with fixed number of particles are invariant for the evolution
associated with the positive self-adjoint operator
$d\Gamma(\omega)$. Furthermore, since $H_{01}\otimes 1$ and $1\otimes
H_{02}$ are positive self-adjoint commuting operators, $H_0$ is a
positive self-adjoint operator with domain $D(H_0)= D(H_{01}\otimes
1)\cap D(1\otimes H_{02})$. Assumption \ref{ass:2} is also satisfied
by usual estimates: $\forall \psi \in \mathscr{H}$, $\forall\phi_n\in
L^2 (\mathds{R}^{pd})\otimes L^2_s(\mathds{R}^{nd})\cap
Q(H_{01}\otimes 1)$, $n\in\mathds{N}$,
\begin{equation}
  \label{eq:13}
  \begin{split}
    \lvert \langle \psi , H_I\phi_n \rangle_{}\rvert\leq \sqrt{2p}\bigl( 2 \sqrt{n} \lVert \lvert v(x,\cdot)(1-\Delta_x)^{-1/2}\rvert_{\mathcal{L}(L^2_{(x)})}^{} \rVert_{L^2_{(k)}} + \lvert (1-\Delta_x)^{-1/2}\lVert v(x,\cdot) \rVert_{L^2_{(k)}}^2\\(1-\Delta_x)^{-1/2}\rvert_{\mathcal{L}(L^2_{(x)})}^{1/2}\bigr)\Bigl(\Bigl\lVert \bigl(\sum_{i=1}^p-\Delta_{x_i}\bigr)^{1/2}\phi_n \Bigr\rVert+\sqrt{p} \lVert \phi_n \rVert\Bigr) \sum_{\substack{i=-1\\i\neq 0}}^{1}\lVert \psi_{n+i} \rVert\; .
  \end{split}
\end{equation}
Then $H_N$ is essentially self-adjoint on $D(H_0)\cap f_0(L^2
(\mathds{R}^{pd})\otimes L^2(\mathds{R}^{d})^{(\cdot)})$.

Let $H_N\rvert_s$ be the restriction of $H_N$ to $L^2_s
(\mathds{R}^{pd})\otimes \Gamma_s(L^2 (\mathds{R}^d))$. It is possible
to extend $H_N\rvert_s$ to $\Gamma_s(L^2 (\mathds{R}^d))\otimes
\Gamma_s(L^2 (\mathds{R}^d))$ in the following way. Define
\begin{equation}
  \label{eq:25}
  \widetilde{H}_N=d\Gamma(-\Delta+V)\otimes 1+1\otimes d\Gamma(\omega)+\int_{\mathds{R}^d}^{}  \psi^{*}(x)\bigl(a^{*}(v(x,\cdot))+a(v(x,\cdot))\bigr)\psi(x)dx\; ,
\end{equation}
where $\psi^{\#}$ are the creation and annihilation operators
corresponding to the first Fock space. Then $H_N\rvert_s$ and
$\widetilde{H}_N$ agree on the $p$-particle sector $L^2_s
(\mathds{R}^{pd})\otimes \Gamma_s(L^2 (\mathds{R}^d))$ of
$\Gamma_s(L^2 (\mathds{R}^d))\otimes \Gamma_s(L^2
(\mathds{R}^d))$. The self-adjointness of $\widetilde{H}_N$ still
follows from Theorem~\ref{thm:1} using the bound~\eqref{eq:13}: it is
sufficient to choose for $\mathscr{H}_1\equiv \Gamma_s(L^2
(\mathds{R}^d))$ the decomposition in finite particle vectors
$\{\mathscr{H}_{1}^{(j)}\otimes \Gamma_s(\mathscr{H}_2)
\}_{j\in\mathds{N}}\equiv \{L^2_s (\mathds{R}^{jd})\otimes
\Gamma_s(\mathscr{H}_2) \}_{j\in\mathds{N}}$. Let
$H_0\equiv d\Gamma(-\Delta+V)\otimes 1 +1\otimes d\Gamma(\omega)$, then the
domain of essential self-adjointness for $\widetilde{H}_N$ is
$D(H_0)\cap f_0(L^2(\mathds{R}^{d})^{(\cdot)}\otimes
L^2(\mathds{R}^{d})^{(\cdot)})$. Let $N_1$ and $N_2$ be the number
operators corresponding to the first and second Fock space
respectively. Then applying Proposition~\ref{prop:1} we also obtain
$D(\widetilde{H}_N)\cap D(N_1^2+N_2^2)=D(H_0)\cap D(N_1^2+N_2^2)$.

\subsection{Pauli-Fierz Hamiltonian.}
\label{sec:pauli-fierz-hamilt}

The last example considered is an operator describing the dynamics of
rigid charges and their radiation field interacting. The model was
introduced by \citet{pauli1938theorie}, and has been extensively
studied by a mathematical standpoint. See \citet[][and references
thereof contained]{MR2097788} for a detailed presentation.

Let $\mathscr{H}^{(spin)}=(\otimes^p
\mathds{C}^{2[\frac{d}{2}]})\otimes L^2 (\mathds{R}^{pd})\otimes
\Gamma_s(\mathds{C}^{d-1}\otimes L^2 (\mathds{R}^d))$, $\mathscr{H}=
L^2 (\mathds{R}^{pd})\otimes \Gamma_s(\mathds{C}^{d-1}\otimes L^2
(\mathds{R}^d))$: the first space corresponds to $p$
spin-$\frac{1}{2}$ particles, the second to spinless particles. Let
$\chi\in L^2 (\mathds{R}^d)$, $V\in L^2_{loc}
(\mathds{R}^{pd},\mathds{R}_{+})$, $\omega=\lvert k\rvert $, $m_j>0$,
$q_j\in \mathds{R}$ for all $j=1,\dotsc,p$. Furthermore, let
$e_{\lambda}: \mathds{R}^d\to \mathds{R}^d$ such that for almost all
$k\in \mathds{R}^d$, $k\cdot e_{\lambda}(k)=0$ and $e_{\lambda}(k)
\cdot e_{\lambda'}(k)=\delta_{\lambda\lambda'}$ for all
$\lambda,\lambda'=1,\dotsc, d-1$. Then we define the electromagnetic
vector potential in the Coulomb gauge as
\begin{equation}
  \label{eq:26}
  A(x)=\sum_{\lambda=1}^{d-1}\int_{\mathds{R}^d}^{}e_{\lambda}(k)\Bigl( a^{*}_{\lambda}(k)\chi(k)e^{ik\cdot x}+a_{\lambda}(k)\bar{\chi}(k)e^{ik\cdot x} \Bigr)  dk\; ;
\end{equation}
where $a_{\lambda}^{\#}$ are the creation and annihilation operators
of $\Gamma_s(\mathds{C}^{d-1}\otimes L^2 (\mathds{R}^d))$ satisfying
the canonical commutation relations
$[a_{\lambda}(k),a^{*}_{\lambda'}(k')]=\delta_{\lambda\lambda'}\delta(k-k')$;
the (spinless) Pauli-Fierz Hamiltonian on $\mathscr{H}$ is then
\begin{equation}
  \label{eq:27}
  H_{PF}= \sum_{j=1}^p \frac{1}{2m_j}\bigl(-i\nabla_j\otimes 1+q_j A(x_j)\bigr)^2 + V(x_1,\dotsc, x_p)\otimes 1 + 1\otimes \sum_{\lambda=1}^{d-1}\int_{\mathds{R}^d}^{}\omega(k)a^{*}_{\lambda}(k)a_{\lambda}(k)  dk\; .
\end{equation}
The function $\chi$ plays the role of an ultraviolet cut off in the
interaction, and is usually interpreted as the
Fourier transform of the particles' charge distribution. Let
$\{\sigma^{(\mu)}\}_{\mu=1}^d$ the $2^{[\frac{d}{2}]}\times
2^{[\frac{d}{2}]}$ matrices satisfying
$\sigma^{(\mu)}\sigma^{(\nu)}+\sigma^{(\nu)}\sigma^{(\mu)}=2\delta_{\mu\nu}\mathrm{Id}$. Also,
denote by $\sigma_j^{(\mu)}$, $j=1,\dotsc, p$ the operator on
$(\otimes^p \mathds{C}^{2[\frac{d}{2}]})$ acting as $\sigma^{(\mu)}$
on the $j$-th space of the tensor product. Then the spin-$\frac{1}{2}$
Pauli-Fierz Hamiltonian on $\mathscr{H}^{(spin)}=(\otimes^p
\mathds{C}^{2[\frac{d}{2}]})\otimes \mathscr{H}$ can be written as:
\begin{equation}
  \label{eq:28}
  H_{PF}^{(spin)}=1\otimes H_{PF} +\frac{i}{2}\sum_{j=1}^pq_j\sum_{1\leq \mu<\nu\leq d}^{}\sigma_j^{(\mu)}\sigma_j^{(\nu)}\otimes \Bigl(\partial_j^{(\mu)}A^{(\nu)}(x_j)-\partial_j^{(\nu)}A^{(\mu)}(x_j)\Bigr)\; ;
\end{equation}
where $A^{(\mu)}(x)$ is the $\mu$-th component of the vector $A(x)$.

The quadratic form corresponding to the Pauli-Fierz Hamiltonian is
bounded from below, so it is possible to define at least one
self-adjoint extension by means of the Friedrichs Extension
Theorem. This type of information is not completely satisfactory,
since infinitely many extensions may exist, each one dictating a
different dynamics for the system. For small values of the ratios
$q^2_j/m_j$ between charge and mass of the particles, and if
$\chi,\chi/\sqrt{\omega}\in L^2 (\mathds{R}^d)$, a unique self-adjoint
extension is given by KLMN Theorem. For arbitrary values of the ratios
$q^2_j/m_j$, it is possible to prove essential self-adjointness of
both $H_{PF}$ and $H_{PF}^{(spin)}$ (for the spin operator we need in
addition $\omega\chi\in L^2 (\mathds{R}^d)$) by means of
Theorem~\ref{thm:1}, under the sole assumption $\chi\in L^2
(\mathds{R}^d)$. As discussed in Section~\ref{sec:introduction}, an
analogous result (on a slightly different domain) has been obtained
with an argument of functional integration by \citet{MR1891842}. If
the dependence on $x$ of $A(x)$ is more general, functional
integration methods may not be applicable; however Theorem~\ref{thm:1}
still holds.

In the following discussion we will focus on a simplified model, for
the sake of clarity. Assumptions~\ref{ass:1} and~\ref{ass:2} are
checked on $H_{PF}$ with $p=1$, $m=1/2$ and $q=-1$, i.e.:
$\mathscr{H}\equiv L^2 (\mathds{R}^d)\otimes
\Gamma_s(\mathds{C}^{d-1}\otimes L^2 (\mathds{R}^d))$ and
\begin{equation}
  \label{eq:29}
  H\equiv \bigl(i\nabla_x\otimes 1+A(x)\bigr)^2 + V(x)\otimes 1 + 1\otimes \sum_{\lambda=1}^{d-1}\int_{\mathds{R}^d}^{}\omega(k)a^{*}_{\lambda}(k)a_{\lambda}(k)  dk\; .
\end{equation}

Observe that, since we are in the Coulomb gauge, $\nabla_x\cdot
A(x)=0$ hence $[-i\nabla_x\otimes 1, A(x)]=0$ on a suitable dense
domain. Rewrite $H$ in the following form, to identify the free and
interaction parts:
\begin{equation}
  \label{eq:30}
  \begin{split}
    H= \bigl(-\Delta_x+V(x)\bigr)\otimes 1+1\otimes \sum_{\lambda=1}^{d-1}\int_{\mathds{R}^d}^{}\omega(k)a^{*}_{\lambda}(k)a_{\lambda}(k)  dk +2i A(x)\cdot(\nabla_x\otimes 1)+ A^2(x)\; .
  \end{split}
\end{equation}
We identify $H_{01}\equiv -\Delta+V$, $H_{02}\equiv
\sum_{\lambda}\int_{\mathds{R}^d}^{}\omega a^{*}_{\lambda}a_{\lambda}$
and $H_I\equiv 2i A\cdot(\nabla\otimes 1)+
A^2$. Assumption~\ref{ass:1} is satisfied, as in the Nelson
model~\eqref{eq:23} above. For the interaction part, we have the
following bounds: $\forall\psi\in\mathscr{H}$, $\forall\phi_n\in L^2
(\mathds{R}^d)\otimes (\mathds{C}^{d-1}\otimes L^2
(\mathds{R}^d))^{\otimes_s n}\cap Q(H_{01}\otimes 1)$,
$n\in\mathds{N}$,
\begin{equation}
  \label{eq:33}
  \begin{aligned}
    \lvert \langle \psi , &A(x)\cdot (\nabla_x\otimes 1)\phi_n\rangle_{}\rvert\leq \sqrt{2(d-1)}\lVert \chi \rVert_2^{} \sqrt{n+1}\lVert (\lvert \nabla_x\rvert_{}^{}\otimes 1) \phi_n\rVert_{}^{}\sum_{\substack{i=-1\\i\neq 0}}^1\lVert \psi_{n+i} \rVert_{}^{}\; ;\\
    \lvert \langle \psi , &A^2(x)\phi_n\rangle_{}\rvert\leq 2(d-1) \lVert \chi \rVert_2^{}(n+1)\lVert \phi_n \rVert_{}^{}\sum_{i=-2}^{2}\lVert \psi_{n+i} \rVert_{}^{}\; .
  \end{aligned}  
\end{equation}
Hence Assumption~\ref{ass:2} is satisfied. Then $H$ is essentially
self-adjoint on $D(H_0)\cap f_0(L^2 (\mathds{R}^{d})\otimes
(\mathds{C}^{d-1}\otimes L^2 (\mathds{R}^d))^{(\cdot)})$.
\begin{remark*}
  \label{rem:1}
  Neither non-negativity of the Pauli-Fierz operator nor smallness of
  the coupling constant are necessary to prove essential
  self-adjointness by means of Theorem~\ref{thm:1}. Using operator
  methods (commutator estimates), self-adjointness of $H_{PF}$ with
  $V=0$ has been proved for general coupling constants
  in~\citep{MR2436496}, but the non-negativity was needed to
  associate a unique self-adjoint operator to the quadratic
  form. Theorem~\ref{thm:1} relies on different assumptions, and takes
  advantage of the fibered structure of the Fock space: boundedness
  from below of the operator is, in general, not necessary. In fact,
  the Hamiltonians considered in Sections~\ref{sec:hamilt-many-bosons}
  and~\ref{sec:nels-type-hamilt} are possibly unbounded from below, as
  well as the following extension~\eqref{eq:31} of the Pauli-Fierz
  Hamiltonian to infinite degrees of freedom (for the particles). As
  outlined in Section~\ref{sec:conclusive-remarks-1}, if we assume
  boundedness from below, Theorem~\ref{thm:1} can be extended to
  operators quartic in the creation/annihilation operators (see
  Assumptions~\ref{ass:5},~\ref{ass:6} and Theorem~\ref{thm:2}).
\end{remark*}

Let $m_j=1/2$, $q_j=-1$ and
$V=\sum_{i=1}^pV_{ext}(x_i)+\sum_{i<j}V_{pair}(x_i-x_j)$ such that
$V_{ext}\in L^2_{loc}(\mathds{R}^d,\mathds{R}_{+})$,
$V_{pair}(x)=V_{pair}(-x)$ and
$V_{pair}(1-\Delta)^{-1/2}\in\mathcal{L}(L^2 (\mathds{R}^d))$. Under
these assumptions define $H_{PF}\rvert_s$ as the restriction
of~\eqref{eq:27} to $L^2_s (\mathds{R}^{pd})\otimes
\Gamma_s(\mathds{C}^{d-1}\otimes L^2 (\mathds{R}^d))$. The physical
interpretation is a system of $p$ identical bosonic charges subjected
to an external potential, interacting via pair interaction and with
their radiation field. As we did for the Nelson model
in~\eqref{eq:25}, we can extend $H_{PF}\rvert_s$ to $\Gamma_s(L^2
(\mathds{R}^d))\otimes \Gamma_s(\mathds{C}^{d-1}\otimes L^2
(\mathds{R}^d))$:
\begin{equation}
  \label{eq:31}
  \begin{split}
    \widetilde{H}_{PF}=\int_{\mathds{R}^d}^{}\psi^{*}(x)\Bigl\{\bigl(i\nabla_x\otimes 1 + A(x)\bigr)^2+V_{ext}(x) \Bigr\}\psi(x)  dx + \frac{1}{2}\int_{\mathds{R}^{2d}}^{}V_{pair}(x-y)\psi^{*}(x)\psi^{*}(y)\\\psi(x)\psi(y)  dxdy +1\otimes \sum_{\lambda=1}^{d-1}\int_{\mathds{R}^d}^{}\omega(k)a^{*}_{\lambda}(k)a_{\lambda}(k)  dk\; .
  \end{split}
\end{equation}
We would like to prove essential self-adjointness by means of
Theorem~\ref{thm:1}. Identify $H_{01}\equiv \int_{}^{}\psi^{*}(-\Delta
+V_{ext})\psi $; $H_{02}\equiv
\sum_{\lambda}\int_{\mathds{R}^d}^{}\omega
a^{*}_{\lambda}a_{\lambda}$; $H_I\equiv \int \psi^{*}(2i
A\cdot(\nabla\otimes 1)+
A^2)\psi+\frac{1}{2}\int_{}^{}V_{pair}\psi^{*}\psi^{*}\psi\psi $; and
$\{\mathscr{H}_{1}^{(j)}\otimes \Gamma_s(\mathscr{H}_2)
\}_{j\in\mathds{N}}\equiv \{L^2_s (\mathds{R}^{jd})\otimes
\Gamma_s(\mathds{C}^{d-1}\otimes L^2 (\mathds{R}^d))
\}_{j\in\mathds{N}}$. Then Assumptions~\ref{ass:1} and~\ref{ass:2} are
satisfied using bounds analogous to~\eqref{eq:33} and~\eqref{eq:35}
(for $V_{pair}$), for each fixed $j\in \mathds{N}$. Hence
$\widetilde{H}_{PF}$ is essentially self-adjoint on $D(H_{01}\otimes
1)\cap D(1\otimes H_{02})\cap f_0(L^2
(\mathds{R}^{d})^{(\cdot)}\otimes (\mathds{C}^{d-1}\otimes L^2
(\mathds{R}^d))^{(\cdot)})$.






\section{Conclusive remarks}
\label{sec:conclusive-remarks-1}

The examples of the preceding section are not exhaustive: we focused
on them because of their relevance in physical and mathematical
literature. The application to operators on curved space-time, or to
anti-symmetric systems may also lead to results of interest.

The Assumptions~\ref{ass:1}, \ref{ass:2} and \ref{ass:3} are easy to
check: in the examples above follow from basic estimates of creation
and annihilation operators. The proof of Theorem~\ref{thm:1} itself is
not complicated, and relies on the direct sum decomposition of
$\Gamma_s(\mathscr{H}_2)$ and the structure of the interaction with
respect to the latter. Hence this criterion gives, in our opinion, a
simple yet powerful tool to prove essential self-adjointness in Fock
spaces, tailored to take maximum advantage of their structure.

If we assume that $H$ is bounded from below, we can take inspiration
from \citet{masson1971} and extend our criterion to accommodate quartic
operators. The modified assumptions and theorem would then read:

\begin{assumption}{B$_{H}$}
  \label{ass:5}
  $H$ is a densely defined symmetric operator on
  $\mathscr{H}=\mathscr{H}_1\otimes \Gamma_s(\mathscr{H}_2)$ bounded
  from below. $H_{01}$ and $H_{02}$ are self-adjoint operators bounded
  from below such that $\forall t\in \mathds{R}$,
  $\{\mathscr{H}_2^{(n)}\}_{n\in\mathds{N}}$ is invariant for $e^{it
    H_{02}}$.
\end{assumption}

\begin{assumption}{B$_I$}
  \label{ass:6}
  $H_I$ is a symmetric operator on $\mathscr{H}$, with a domain of
  definition $D(H_I)$ such that $D(H_0)\cap D(H_I)$ is dense in
  $\mathscr{H}$. Furthermore exists a complete collection
  $\{\mathscr{H}_{1}^{(j)}\otimes \Gamma_s(\mathscr{H}_2)
  \}_{j\in\mathds{N}}$ invariant for $H_0$ and $H_I$ such that:
  $\forall\phi\in Q(H_{01}\otimes 1)\cap Q(1\otimes H_{02})\cap
  \mathscr{H}_1^{(j)}\otimes \mathscr{H}_2^{(n)}$,
  \begin{equation}
    \label{eq:34}
    H_I\,\phi\in \bigoplus_{i=-4}^{4}\mathscr{H}_1^{(j)}\otimes \mathscr{H}_2^{(n+i)}\; .
  \end{equation}
  Also, $H_I$ satisfies the following bound: $\forall
  j,n\in\mathds{N}$ $\exists C(j)>0$ such that $\forall\psi\in
  \mathscr{H}$, $\forall\phi\in Q(H_{01}\otimes 1)\cap Q(1\otimes
  H_{02})\cap\mathscr{H}_1^{(j)}\otimes \mathscr{H}_2^{(n)}$:
  \begin{equation}
    \label{eq:37}
    \begin{split}
      \lvert \langle \psi , H_I\phi \rangle_{\mathscr{H}}\rvert_{}^2\leq C^2(j)\sum_{i=-4}^4\lVert \psi_{j,n+i} \rVert_{\mathscr{H}_1^{(j)}\otimes\mathscr{H}_2^{(n+i)}}^2 \Bigl[(n+1)^4\lVert \phi \rVert_{\mathscr{H}_1^{(j)}\otimes\mathscr{H}_2^{(n)}}^{2}+(n+1)^{2}\Bigl(q_{H_{01}\otimes 1}(\phi,\phi)\\+q_{1\otimes H_{02}}(\phi,\phi)+(\lvert M_1\rvert_{}^{}+\lvert M_2\rvert_{}^{}+1)\lVert \phi \rVert_{\mathscr{H}_1^{(j)}\otimes\mathscr{H}_2^{(n)}}^{2} \Bigr)\Bigr]\; .
    \end{split}
  \end{equation}
\end{assumption}

\begin{thm}
  \label{thm:2}
  Assume \ref{ass:5} and \ref{ass:6}. Then $H$ is essentially self
  adjoint on $D(H_{01}\otimes 1)\cap D(H_{02}\otimes 1)\cap
  f_0(\mathscr{H}_1^{(\cdot)}\otimes \mathscr{H}_2^{(\cdot)})$.
\end{thm}

\begin{remark*}
  An attempt to extend the results of \citep{masson1971} can be
  found in \citep{arai1991}. Theorem~\ref{thm:2} is a generalization
  of both: it can be applied to more singular situations and a more
  general class of spaces.
\end{remark*}

The proof of Theorem~\ref{thm:1} can be adapted to
Theorem~\ref{thm:2}, making use of the inferior bound for $H$.  We
remark that Assumption~\ref{ass:5}, by itself, implies that $H$ has at
least one self-adjoint extension: it may be tricky to prove for
general operators. Theorem~\ref{thm:2} essentially states that for
regular enough quartic interactions, existence of a particular
self-adjoint extension (the Friedrichs one) is equivalent to its
uniqueness. It may have interesting applications in CQFT: e.g. the
$d$-dimensional (bounded from below) $Y_d$ and
$(\lambda\varphi(x)^4)_d$ models with cut offs have interactions that
are at most quartic and regular. It is our hope that the ideas
utilized in this paper could contribute to improve the mathematical
insight on interacting quantum field theories, and could be developed
to study self-adjointness of more singular systems.

\begin{acknowledgments}
  This work has been supported by the Centre Henri Lebesgue (programme
  ``Investissements d'avenir'' --- ANR-11-LABX-0020-01). The author
  would like to thank Giorgio Velo, that has suggested to him the idea
  of a direct proof of self-adjointness on Fock spaces. Also, he would
  like to thank Zied Ammari and Francis Nier for precious advices and
  stimulating discussions during the redaction of the paper.
\end{acknowledgments}


\end{document}